\documentclass[a4paper,reqno,11pt]{amsart}
\usepackage[a4paper,margin=0.8in]{geometry}
\usepackage{graphicx}
\usepackage{amsmath,amssymb,amsthm}
\usepackage{hyperref}

\theoremstyle{plain}
\newtheorem{theorem}{Theorem}[section]
\newtheorem{lemma}[theorem]{Lemma}
\newtheorem{proposition}[theorem]{Proposition}
\newtheorem{definition}{Definition}[section]
\newtheorem{corollary}[theorem]{Corollary}
\theoremstyle{definition}

\numberwithin{equation}{section}

\newcommand{\R}{\mathbb{R}}
\newcommand{\abs}[2][]{#1\lvert #2 #1\rvert}

\newcommand{\FF}{{\mathcal S}}

\newcommand{\ws}{w^{(s)}}
\newcommand{\dmin}{d_-(r)}
\newcommand{\FFplus}{\FF_+(r)}
\newcommand{\FFmin}{\FF_-(r)}

\newcommand{\etamin}{\check{\eta}}

\newcommand{\wspm}{w^{(s_\pm)}}
\newcommand{\wtilde}{\tilde{w}}
\newcommand{\phis}{\Phi^{(s)}}
\newcommand{\phistar}{\Phi^{(s_\star)}}

\title{A sharp version of the Benjamin and Lighthill conjecture for steady waves with vorticity}
\author{Evgeniy Lokharu}
\address{Department of Mathematics, Link\"oping University, SE-581 83 Link\"oping, Sweden}

\keywords{Steady waves, vorticity, gravity waves, flow force}

\begin{document}
	
\begin{abstract}
	
We prove the Benjamin and Lighthill conjecture for all two-dimensional steady water waves with an arbitrary vorticity distribution. We show that the flow force constant of an arbitrary smooth wave is bounded by the corresponding flow force constants for conjugate laminar flows. We prove these inequalities without any assumptions on the geometry of the surface profile and put no restrictions on wave's amplitude. Furthermore, we give a complete description of cases when equalities can occur.  Our results are new already for Stokes waves with vorticity, while the case of equalities is new even in the irrotational setting. Beside proving the Benjamin and Lighthill conjectrure, we establish sharp bounds for the surface profile, extending previous results on two-dimensional steady water waves.
	
\end{abstract}

\maketitle

\section{Introduction} \label{s:introduction}

In this paper we address the classical water wave problem for two-dimensional steady waves with vorticity on water of finite depth, formulated in terms of Euler equations with a free boundary. While allowing for arbitrary exact solutions, representing nonlinear waves, we focus
on two questions related to fundamental bounds for the surface profile and a possible range for values of the flow force constant. The latter problem  known as Benjamin and Lighthill conjecture was introduced in \cite{Benjamin1954a}. The property conjectured by Benjamin and Lighthill (see also Keady \& Norbury \cite{Keady1975}, Conjecture 2) can be expressed as inequalities
\begin{equation} \label{BLconj}
\FF_-(r) \leq \FF \leq \FF_+(r),
\end{equation}
where $\FF$ is the flow force constant of a solution, $r$ is the corresponding Bernoulli constant, while $\FF_-(r)$ and $\FF_+(r)$ are flow force constants of conjugate laminar flows (supercritical and subcritical
respectively) determined by the same Bernoulli constant. According to the conjecture, inequality \eqref{BLconj} is valid for arbitrary smooth solutions. It was verified by Benjamin \cite{Benjamin95} for all
irrotational Stokes waves and their small perturbations, while the bottom bound in \eqref{BLconj} was obtained earlier by Keady \& Norbury \cite{Keady1975} (also for periodic wavetrains). Kozlov and Kuznetsov
\cite{Kozlov2009a,Kozlov2011a} proved \eqref{BLconj} for arbitrary solutions under weak regularity assumptions, provided the
Bernoulli constant $r$ is close to it's critical value $R_c$; it was extended to the rotational setting in \cite{Kozlov2017a}, again for $r \approx R_c$; the latter condition guarantees that solutions are of small amplitude. The left inequality in \eqref{BLconj} for periodic waves with a favorable vorticity was obtained by Keady \& Norbury \cite{Keady1978}. Whereas their result is valid under essential restrictions on the vorticity, they point out that in general the statement of the conjecture is probably false: "There is no reason to suppose that the conjectures of Benjamin and Lighthill will hold for all flows with vorticity". Thus, it is especially surprising that \eqref{BLconj} turns out to be true for arbitrary vorticity distributions and arbitrary solutions, which is one the main results of the present paper. 

Let us outline some difficulties and gaps associated with \eqref{BLconj}. Even so Benjamin \cite{Benjamin95} verified \eqref{BLconj} for all irrotational Stokes waves by estimating certain contour integrals using the divergence structure of the problem, their approach can hardly be extended further and it does not explain the nature of inequalities in \eqref{BLconj}. Furthermore, nonlinear waves on water of finite depth are
not limited to Stokes and solitary waves; see \cite{Vanden_Broeck_1983, Baesens1992, Zufiria1987, Craig2002}. In Section 5.3 of \cite{Benjamin95} Benjamin discusses a possibility to extend their method to arbitrary solutions, however the suggested argument depends heavily on the geometry of the flow (symmetry, monotonicity) and is not applicable in general. While for irrotational Stokes waves inequalities in \eqref{BLconj} are strict, the case of equalities becomes a significant problem in the context of arbitrary solutions. It is known that $\FF = \FF_-$ for all solitary waves, while it is unclear if the equality $\FF = \FF_-$ holds true only for solitary waves, symmetric and monotone on each side around the crest.

The Benjamin and Lighthill conjecture is closely related with another problem about bounds for the surface profile. If $y = \eta(x)$ determines the surface of the fluid in a moving frame of reference, then the following inequalities are well known:
\begin{equation} \label{profile_bounds}
	d_-(r) \leq \inf_{x \in \R} \eta(x) \leq d_+(r) \leq \sup_{x \in \R} \eta(x),
\end{equation}
where $d_-(r)$ and $d_+(r)$ are depths of the supercritical and subcritical flows respectively. First obtained by Keady \& Norbury \cite{KeadyNorbury78} for irrotational Stokes waves, it was extended to
arbitrary solutions by Kozlov and Kuznetsov \cite{Kozlov2007, Kozlov2009}; see also \cite{Kozlov2012}. We emphasize that Kozlov
and Kuznetsov \cite{Kozlov2009} obtained strict inequality $d_+(r) < \sup_{x\in\R} \eta(x)$ for arbitrary irrotational
solutions, provided $\eta$ is not a constant identically. While for Stokes waves it can be obtained by using the Hopf lemma, the general case is much more subtle. The argument in \cite{Kozlov2009} required a careful analysis of the Fourier symbol associated with an integro-differential operator and the irrotational nature of the problem was essential. For waves with vorticity only a weak form of \eqref{profile_bounds} is known; see \cite{Kozlov2012, Kozlov2015}.

In this paper we consider both problems, inequalities for the flow force \eqref{BLconj} and bounds \eqref{profile_bounds}. For an arbitrary wave with vorticity we prove \eqref{BLconj} and \eqref{profile_bounds} and provide a complete description of all cases when equalities can occur. The case of equalities in \eqref{BLconj} is new even in the irrotational setting. In fact, for all solutions other than streams and classical solitary waves all inequalities in \eqref{BLconj} and \eqref{profile_bounds} are shown to be strict. In particular, if a given solution satisfies $\FF = \FF_-(r)$ (without any assumptions on the surface profile), then it is necessarily a classical solitary wave of elevation, whose profile decays monotonically on each side of the crest. On the other hand, the relation $\FF = \FF_+(r)$ is only valid for subcritical laminar flows. This is a strong statement, because it, in particular, shows that subcritical solitary waves (with the Froude number less than one) do not exist. The latter was an open problem
for a long time even in the irrotational setting; it was resolved recently in \cite{KozLokhWheeler2020} by using an asymptotic analysis. In addition we prove that any steady wave is subject to strict inequalities
in \eqref{profile_bounds}, provided it is not a parallel flow or a solitary wave, for which the left inequality in \eqref{profile_bounds}
turns into the equality. This generalizes a series of previous results.

It is remarkable that our argument is based essentially on the classical maximum principle applied for a version of the flow force flux function, introduced recently in \cite{KozLokhWheeler2020}. Thus, the Benjamin and Lighthill conjecture \eqref{BLconj} is basically a consequence of the elliptic maximum principle, which explains the nature of \eqref{BLconj}. 

\section{Statement of the problem}

We consider the classical water wave problem for two-dimensional steady waves with vorticity on water of finite depth. We neglect effects of surface tension and consider a fluid of constant (unit) density. Thus, in an appropriate coordinate system moving along with the wave, stationary Euler equations are given by
\begin{subequations}\label{eqn:trav}
	\begin{align}
	\label{eqn:u}
	(u-c)u_x + vu_y & = -P_x,   \\
	\label{eqn:v}
	(u-c)v_x + vv_y & = -P_y-g, \\
	\label{eqn:incomp}
	u_x + v_y &= 0, 
	\end{align}
	which holds true in a two-dimensional fluid domain $D$, defined by the inequality
	\[
	0 < y < \eta(x).
	\]
	Here $(u,v)$ are components of the velocity field, $y = \eta(x)$ is the surface profile, $c$ is the wave speed, $P$ is the pressure and $g$ is the gravitational constant. The corresponding boundary conditions are
	\begin{alignat}{2}
	\label{eqn:kinbot}
	v &= 0&\qquad& \text{on } y=0,\\
	\label{eqn:kintop}
	v &= (u-c)\eta_x && \text{on } y=\eta,\\
	\label{eqn:dyn}
	P &= P_{\mathrm{atm}} && \text{on } y=\eta.
	\end{alignat}
\end{subequations}
It is often assumed in the literature that the flow is irrotational, that is $v_x - u_y$ is zero everywhere in the fluid domain. Under this assumption components of the velocity field are harmonic functions, which allows to apply methods of complex analysis. Being a convenient simplification it forbids modeling of non-uniform currents, commonly occurring in nature. In the present paper we will consider rotational flows, where the vorticity function is defined by
\begin{equation} \label{vort}
\omega = -v_x + u_y.
\end{equation}
Throughout the paper we assume that the flow is free from stagnation points and the horizontal component of the relative velocity field does not change sign, that is
\begin{equation} \label{uni}
u-c < 0 
\end{equation}
everywhere in the fluid. We call such flows unidirectional.

In the two-dimensional setup relation \eqref{eqn:incomp} allows to reformulate the problem in terms of a stream function $\psi$, defined implicitly by relations
\[
\psi_y = c-u, \ \ \psi_x =  v.
\]
This determines $\psi$ up to an additive constant, while relations \eqref{eqn:kinbot},\eqref{eqn:kinbot} force $\psi$ to be constant along the boundaries. Thus, by subtracting a suitable constant, we can always assume that
\[
\psi = m, \ \ y = \eta; \ \ \psi = 0, \ \ y = 0.
\]
Here $m > 0$ is the mass flux, defined by
\[
m = \int_0^\eta (c-u) dy.
\]
In what follows we will use non-dimensional variables proposed by Keady \& Norbury \cite{KeadyNorbury78}, where lengths and velocities are scaled by $(m^2/g)^{1/3}$ and $(mg)^{1/3}$ respectively; in new units $m=1$ and $g=1$. For simplicity we keep the same notations for $\eta$ and $\psi$.

Taking the curl of Euler equations \eqref{eqn:u}-\eqref{eqn:incomp} one checks that the vorticity function $\omega$ defined by \eqref{vort} is constant along paths tangent everywhere to the relative velocity field $(u-c,v)$; see \cite{Constantin11b} for more details. Having the same property by the definition, stream function $\psi$ is strictly monotone by \eqref{uni} on every vertical interval inside the fluid region. These observations together show that $\omega$ depends only on values of the stream function, that is
\[
\omega = \omega(\psi).
\]
This property and Bernoulli's law allow to express the pressure $P$ as
\begin{align}
\label{eqn:bernoulli}
P-P_\mathrm{atm} + \frac 12\abs{\nabla\psi}^2 + y  + \Omega(\psi) - \Omega(1) = const,
\end{align}
where 
\begin{align*}
\Omega(\psi) = \int_0^\psi \omega(p)\,dp
\end{align*}
is a primitive of the vorticity function $\omega(\psi)$. Thus, we can eliminate the pressure from equations and obtain the following problem:
\begin{subequations}\label{eqn:stream}
	\begin{alignat}{2}
	\label{eqn:stream:semilinear}
	\Delta\psi+\omega(\psi)&=0 &\qquad& \text{for } 0 < y < \eta,\\
	\label{eqn:stream:dyn}
	\tfrac 12\abs{\nabla\psi}^2 +  y  &= r &\quad& \text{on }y=\eta,\\
	\label{eqn:stream:kintop} 
	\psi  &= 1 &\quad& \text{on }y=\eta,\\
	\label{eqn:stream:kinbot} 
	\psi  &= 0 &\quad& \text{on }y=0.
	\end{alignat}
\end{subequations}
Here $r>0$ is referred to as Bernoulli's constant. 

Let us define the flow force constant, another motion invariant. Following Benjamin \cite{BENJAMIN1984}, we put
\begin{equation} \label{FFbyP}
\FF = \int_0^\eta (P - P_{atm} + (u-c)^2) dy.
\end{equation}
Taking $x$-derivative in \eqref{FFbyP} and using \eqref{eqn:u} together with the formula for the pressure \eqref{eqn:bernoulli}, one verifies that $\FF$ is a constant of motion independent of $x$. In terms of the stream function one obtains
\begin{equation} \label{flowforce}
\FF = \int_0^{\eta}(\tfrac12(\psi_y^2 - \psi_x^2) - y + \Omega(1) - \Omega(\psi) + r )\, dy.
\end{equation}
This constant is important in several ways; for instance, it plays the role of the Hamiltonian in spatial dynamics; see \cite{Baesens1992}. 

\subsection{Stream solutions} Laminar flows or shear currents, for which the vertical component $v$ of the velocity field is zero play an important role in the theory of steady waves. Let us recall some basic facts about stream solutions $\psi = U(y)$ and $\eta = d$, describing shear currents. It is convenient to parameterize the latter solutions by the relative speed at the bottom. Thus, we put $U_y(0) = s$ and find that $U = U(y;s)$ is subject to
\begin{equation} \label{eqn:laminar}
U'' + \omega(U) = 0, \ \ \ 0 < y < d; \ \ U(0) = 0, \ \ U(d) = 1.
\end{equation}
Our assumption \eqref{uni} implies $U' > 0$ on $[0; d]$, which puts a natural constraint on $s$. Indeed, multiplying the first equation in \eqref{eqn:laminar} by $U'$ and integrating over $[0; y]$, we find
\[
U'^2 = s^2 - 2\Omega(U).
\]
This shows that the expression $s^2 - 2 \Omega(p)$ is positive for all $p \in [0; 1]$, which requires
\[
s > s_0 = \sqrt{\max_{p \in [0,1]}2\Omega(p)}.
\]
On the other hand, every $s > s_0$ gives rise to a monotonically increasing function $U(y; s)$ solving \eqref{eqn:laminar} for some unique $d = d(s)$, given explicitly by
\[
d(s) = \int_0^1 \frac{1}{\sqrt{s^2 - 2\Omega(p)}}.
\]
This formula shows that $d(s)$ monotonically decreases to zero with respect to $s$ and takes values between zero and 
\[
d_0 = \lim_{s \to s_0+} d(s).
\]
The latter limit can be finite or not. For instance, when $\omega = 0$ we find $s_0 = 0$ and $d_0 = +\infty$. On the other hand, when $\omega = -b$ for some positive constant $b \neq 0$, then $s_0 = 0$ but $d_0 < + \infty$. We note that our main theorem is concerned with the case $d_0 < + \infty$.

Every stream solution $U(y;s)$ determines the Bernoulli constant $R(s)$, which can be found from the relation \eqref{eqn:stream:kintop}. This constant can be computed explicitly as
\[
R(s) = \tfrac12 s^2 - \Omega(1) + d(s).
\]
As a function of $s$ it decreases from $R_0$ to $R_c$ when $s$ changes from $s_0$ to $s_c$ and increases to infinity for $s>s_c$. Here the critical value $s_c$ is determined by the relation
\[
\int_0^1 \frac{1}{(s^2 - 2 \Omega(p))^{3/2}} dp = 1.
\]
The latter monotonicity property of $R(s)$ shows (see Figure 1) that for any $r \in (R_c;R_0)$ the equation $R(s) = r$ has exactly two solutions $s = s_-(r)$ and $s = s_+(r)$, such that $s_-(r) < s_c < s_+(r)$. The corresponding depths 
\[
d_-(r) = d(s_+(r)), \ \ d_+(r) = d(s_-(r))
\]
satisfy $d_-(r) < d_+(r)$ and are called supercritical and subcritical depths respectively. The flow force constants corresponding to depths $d_\pm(r)$ are denoted by $S_\pm(r)$.

Stream solutions $U(y; s_-(r))$ and $U(y; s_+(r))$ are said to be conjugate and are defined only under condition $r < R_0$. This assumption is naturally fulfilled for all irrotational waves, since
then $R_0 = +\infty$. Under certain assumptions on the vorticity it is shown in \cite{Kozlov2015} that no unidirectional
waves exist for $r > R_0$, except laminar flows. Furthermore, it was verified recently in \cite{Lokharu2020} for all vorticity distributions that solitary waves are absent for $r > R_0$. Thus, the assumption $r<R_0$ appears to be natural in what follows. The bottom bound $r > R_c$ is well known for arbitrary unidirectional waves with vorticity; see \cite{Kozlov2015}, \cite{Kozlov2017a} and \cite{Wheeler15b}.

\subsection{Formulations of main results.}

Following notations from the previous section, our main
theorem is

\begin{theorem} \label{thm:BLC}
	Let $\psi \in C^{2,\gamma}(\overline{D})$ and $\eta \in C^{2,\gamma}(\R)$ be a solution to \eqref{eqn:stream} with $\inf_{D}\psi_y > 0$ and $r \in (R_c,R_0)$. Then the flow force constant $\FF$ given by \eqref{flowforce} enjoys the following properties:
	\begin{itemize}
		\item[(i)] inequalities $\FF_-(r) \leq \FF \leq \FF_+(r)$ are always true;
		\item[(ii)] the equality $\FF = \FF_-(r)$ holds true only for supercritical laminar flows and symmetric solitary waves of positive elevation supported by supercritical streams;
		\item[(iii)] the equality $\FF = \FF_+(r)$ is true only for subcritical laminar flows.
	\end{itemize}
\end{theorem}

The claim (i) of Theorem \ref{thm:BLC} is known as the classical Benjamin and Lighthill conjecture. We emphasise that it is new even in the irrotational case since it covers all possible smooth solutions; the original proof by Benjmain \cite{Benjamin95} deals only with Stokes waves (and small amplitude perturbations of those) and Benjmain's argument relies heavily on that assumption.

The parts (ii) and (iii) of Theorem \ref{thm:BLC} are of separate interest and had never been considered in the literature before. For instance, the claim (iii), in particular, forbids the existence of subcritical solitary waves; this was an open problem for a long time and was recently proved in
\cite{KozLokhWheeler2020} using an asymptotic analysis. Both statements (ii) and (iii) are new even for irrotational waves.

\begin{corollary}
	Under assumptions of the theorem, the water wave profile $\eta$ is subject to the following properties:
	\begin{itemize}
		\item[(i')] for all $x\in\R$ we have $d_-(r) < \eta(x)$;
		\item[(ii')] denoting $\hat{\eta} = \sup_{\R} \eta$ and $\check{\eta} = \inf_{\R} \eta$, we have $\check{\eta} < d_+(r) < \hat{\eta}$, while equalities $\check{\eta} = d_+(r)$ or $d_+(r) = \hat{\eta}$ are only possible if $\check{\eta} = \hat{\eta} = d_+(r)$;
		\item[(iii')] the equality $\check{\eta} = d_-(r)$ is valid only for supercritical laminar flows and supercritical solitary waves.
	\end{itemize}
\end{corollary}

These statements follow from Proposition \ref{p:bounds} and Theorem \ref{thm:BLC}. Inequalities $d_-(r) < \eta(x)$ and $\check{\eta} < d_+(r) < \hat{\eta}$ for irrotational Stokes waves were first obtained by Keady \& Norbury \cite{Keady1975}. An extension to arbitrary irrotational solutions was done by Kozlov \& Kuznetsov \cite{Kozlov2007,Kozlov2009}. For waves with
vorticity only non-strict versions of inequalities were known; see \cite{Kozlov2015} and references therein. The last claim (iii') is new even in the irrotational setting.

Our proofs are based on properties of flow force flux functions introduced recently in \cite{KozLokhWheeler2020} and used in \cite{Lokharu2020} for proving the nonexistence of steady waves with $r \geq R_0$.

\section{Preliminaries}

\subsection{Reformulation of the problem}

Under assumption \eqref{uni} we can apply the partial hodograph transform introduced by Dubreil-Jacotin \cite{DubreilJacotin34}.Thus, we present new independent variables
\[
q = x, \ \ p = \psi(x,y),
\]
while new unknown function $h(q,p)$ (height function) is defined from the identity
\[
h(q,p) = y.
\]
Note that it is related to the stream function $\psi$ through the formulas
\begin{equation} \label{height:stream}
\psi_x = - \frac{h_q}{h_p}, \ \ \psi_y = \frac{1}{h_p},
\end{equation}
where 
\begin{subequations}\label{height}
\begin{equation} \label{unih}
h_p > 0
\end{equation}
throughout the fluid domain by \eqref{uni}. An advantage of using new variables is in that instead of two unknown functions $\eta(x)$ and $\psi(x,y)$ with an unknown domain of definition, we have one function $h(q,p)$ defined in a fixed strip $S = \R \times (0,1)$. An equivalent problem for $h(q,p)$ is given by
	\begin{alignat}{2}
	\label{height:main}
	\left( \frac{1+h_q^2}{2h_p^2} + \Omega \right)_p - \left(\frac{h_q}{h_p}\right)_q &=0 &\qquad& \text{in } S,\\
	\label{height:top}
	\frac{1+h_q^2}{2h_p^2} +  h  &= r &\quad& \text{on }p=1,\\
	\label{height:bot} 
	h  &= 0 &\quad& \text{on }p=0.
	\end{alignat}
\end{subequations}
The wave profile $\eta$ becomes the boundary value of $h$ on $p = 1$:
\[
h(q,1) = \eta(q), \ \ q \in \R.
\]

Using \eqref{height:stream} and Bernoulli's law \eqref{eqn:bernoulli} we recalculate the flow force constant $\FF$ defined in \eqref{flowforce} as
\begin{equation}\label{height:ff}
\FF = \int_0^1 \left( \frac{1-h_q^2}{2h_p^2} - h - \Omega + \Omega(1) + r \right) h_p \, dp.
\end{equation}
Laminar flows defined by stream functions $U(y; s)$ correspond to height functions $h = H(p; s)$ that are independent of horizontal variable $q$. The corresponding equations are
\[
\left(\frac{1}{2H_p^2} + \Omega\right)_p = 0, \ \ H(0) = 0, \ \ H(1) = d(s), \ \ \frac{1}{2 H_p^2(1)} + H(1) = R(s).
\]
Solving equations for $H(p; s)$ explicitly, we find
\[
H(p;s) = \int_0^p \frac{1}{\sqrt{s^2 -2\Omega(\tau)}} \, d\tau.
\]
Given a height function $h(q,p)$ and a stream solution $H(p;s)$, we define
\begin{equation}\label{ws}
\ws(q,p) = h(q,p) - H(p;s).
\end{equation}
This notation will be frequently used in what follows. In order to derive an equation for $\ws$ we first write \eqref{height:main} in a non-divergence form as
\[
\frac{1+h_q^2}{h_p^2} h_{pp} - 2\frac{h_q}{h_p} h_{qp} + h_{qq} - \omega(p) h_p = 0.
\]
Now using our ansats \eqref{ws}, we find
\begin{equation}\label{ws:main}
\frac{1+h_q^2}{h_p^2} \ws_{pp} - 2\frac{h_q}{h_p} \ws_{qp} + \ws_{qq} - \omega(p) \ws_p + \frac{(\ws_q)^2 H_{pp}}{h_p^2} - \frac{\ws_p (h_p + H_p) H_{pp}}{h_p^2 H_p^2} = 0.
\end{equation}
Thus, $\ws$ solves a homogeneous elliptic equation in $S$ and is subject to a maximum principle; see \cite{Vitolo2007} for an elliptic maximum principle in unbounded domains. The boundary conditions for $\ws$ can be obtained directly from \eqref{height:top} and \eqref{height:bot} by inserting
\eqref{ws} and using the corresponding equations for $H$. This gives
\begin{subequations}\label{ws:boundary}
\begin{alignat}{2}
\frac{(\ws_q)^2}{2h_p^2} - \frac{\ws_p (h_p + H_p)}{2h_p^2 H_p^2} + \ws &=r- R(s) &\qquad& \text{for } p=1,\label{ws:top} \\ 
\ws &= 0&\qquad& \text{for } p=0. \label{ws:bot}
\end{alignat}
\end{subequations}
For $s = s_\pm(r)$, we have $r - R(s_\pm(r)) = 0$ and \eqref{ws:top} turns into
\begin{equation} \label{wspm}
\frac{\wspm_p}{H_p^3} - \wspm = \frac{(\wspm_q)^2}{2h_p^2} + \frac{(\wspm_p)^2(2h_p + H_p)}{2H_p^3 h_p^2}, \ \ p = 1.
\end{equation}
This shows that $\wspm_p(q; 1)$ is positive whenever $\wspm(q; 1)$ is positive. This property will be used in what follows.

I many formulas, such as \eqref{ws:top}, it is often convenient to omit the dependence on $s$ in the notation of $H$. The right choice of $s$ will be always clear from the context and is the same as for $\ws$. Furthermore, since the Bernoulli constant $r$ will remain unchanged, we will often omit it from notations, such as $s_\pm$ or $\FF_\pm$. 

\subsection{Subsolutions} Let $h \in C^{2,\gamma}(\overline{S})$ be a solution to \eqref{height} for some $r > 0$ and let $\FF$ be the
corresponding flow force constant. For an arbitrary sequence $\{q_j\}_{j=1}^\infty \subset \R$, possibly unbounded, we consider horizontal shifts
\[
h_j(q,p) = h(q+q_j,p), \ \ j \geq 1.
\]
Thus, every function $h_j$ solve the same problem \eqref{height} with the same Bernoulli constant. Now let $\gamma' \in (0,\gamma)$ be given. Then the embedding $C^{2,\gamma}(K) \hookrightarrow C^{2,\gamma'}(K)$ is compact for any compact subset $K \subset \overline{S}$. Because the norms $\|h_j\|_{C^{2,\gamma}(\overline{S})}$ are uniformly bounded in $j$, we can find a subsequence $\{h_{j_k}\}_{k=1}^\infty$ and a function $\tilde{h} \in C^{2,\gamma'}(\overline{S})$ with the following property: for any compact $K \subset \overline{S}$ restrictions of functions $h_{j_k}$ to $K$ converge to $\tilde{h}|_K$ in $C^{2,\gamma'}(K)$. Then it is straightforward to show that $\tilde{h}$ has the same regularity as $h$, that is $\tilde{h} \in C^{2,\gamma}(\overline{S})$. It is also clear that $\tilde{h}$ solves the same elliptic problem \eqref{height} with the same Bernoulli constant $r$ and has the same flow force constant $\FF$. Note that if the convergence takes place for some $\gamma'$ then it is true for all $\gamma' \in (0,\gamma)$ by the interpolation. Such function $\tilde{h}$ will be referred to as a subsolution of $h$. This terminology will be useful in order to avoid multiple repetitions of the argument with subsequences as above. Let us give an explicit definition.

\begin{definition} \label{defsub}
	Given two functions $h,\tilde{h} \in C^{2,\gamma}(\overline{S})$ we say that $\tilde{h}$ is a subsolution of $h$ if there exists a sequence  $\{q_j\}_{j=1}^\infty \subset \R$ such that functions $h_j(q, p) = h(q +q_j , p)$ converge to $\tilde{h}$ in
	$C^{2,\gamma'}(K)$ for all compact $K \subset \overline{S}$ and all $\gamma' \in (0,\gamma)$.
\end{definition}

Note that the definition is symmetric: $\tilde{h}$ is a subsolution of $h$ if and only if $h$ is a subsolution of $\tilde{h}$. The following property of subsolutions will be useful in what follows.

\begin{proposition} \label{subsub}
	Let $\tilde{h}\in C^{2,\gamma}(\overline{S})$ be a subsolution of $h\in C^{2,\gamma}(\overline{S})$ and $\hat{h}\in C^{2,\gamma}(\overline{S})$ is a subsolution of $\tilde{h}$. Then $\hat{h}$ is a subsolution of $h$.
\end{proposition}
\begin{proof}
	Let us fix $\epsilon>0$ and a bounded closed interval $I$. Then, because $\tilde{h}$ is a subsolution of $\hat{h}$, then $\hat{h}$ is a subsolution of $\tilde{h}$ and  one finds $\hat{q} \in \R$ such that the function $\tilde{h}(q+\hat{q})$ is close to $\hat{h}(q)$ on $I$ in $C^{2,\gamma'}(I \times [0,1])$, that is
	\[
	\|\tilde{h}(\cdot+\hat{q}) - \hat{h}(\cdot)\|_{C^{2,\gamma'}(I \times [0,1])} < \epsilon/2.
	\]
	Now we consider functions $\tilde{h}(\cdot+\hat{q})$ and $h$. It is clear that $h$ is a subsolution of $\tilde{h}(\cdot+\hat{q})$ and a similar argument gives $\tilde{q} \in \R$ such that
	\[
	\|h(\cdot+\tilde{q}) - \tilde{h}(\cdot+\hat{q})\|_{C^{2,\gamma'}(I \times [0,1])} < \epsilon/2.
	\]
	Combining two inequalities together, we conclude that for any $\epsilon > 0$ and any interval $I$ there exists $\tilde{q} \in \R$ such that
	\[
	\|h(\cdot+\tilde{q}) - \hat{h}(\cdot)\|_{C^{2,\gamma'}(I \times [0,1])} < \epsilon.
	\]	
	This shows that $\hat{h}$ is a subsolutions of $h$.
\end{proof}

\subsection{General bounds for solutions} As noted by Keady and Norbury \cite{KeadyNorbury78}, the bounds 
\[
d_-(r) \leq \check{\eta} \leq d_+(r) \leq \hat{\eta}
\]
are closely related to the Benjamin and Lighthill conjecture about the 
ow force. Here
\[
\check{\eta} = \inf_\R \eta, \ \ \hat{\eta} = \sup_\R \eta.
\]
For periodic waves all inequalities are strict and the case of equalities is quite delicate and is indirectly contained in claims (ii) and (iii) of Theorem \ref{thm:BLC}. 

Let us recall a precise statement about bounds for the surface profile, mainly borrowed from \cite{Kozlov2015}, that will be used in our proofs.

\begin{proposition} \label{p:bounds} Let $(\psi,\eta)$ be as in Theorem \ref{thm:BLC} with $r \in (R_c,R_0)$. Then the following is true:
	\begin{itemize}
		\item[(i)] for all $x\in\R$ we have $\eta(x) > d_-(r)$; furthermore, if $\check{\eta} = d_-(r)$, then $\FF = \FF_-(r)$;
		\item[(ii)] $\hat{\eta} \geq d_+(r)$ and if the equality holds true, then $\FF = \FF_+(r)$;
		\item[(iii)] $\check{\eta} \leq d_+(r)$ and $\check{\eta} = d_+(r)$ implies $\FF = \FF_+(r)$.
	\end{itemize}
\end{proposition}
\begin{proof} The part of the statement about bounds for the surface profile was proved in \cite{Kozlov2015} and we only need to verify the remaining part about the flow force constant. Let $h$ be the height function corresponding to $\psi$,  defined in Section 3.1. 
	
Assume that $\check{\eta} = d_-(r)$. Then there exists an unbounded sequence $\{q_j\}_{j=1}^\infty$ such that $w^{(s_+)}(q_j,1) \to 0$, where $w^{(s_+)}$ is defined by \eqref{ws} with $s = s_+(r)$. Let $\tilde{w}$ be the corresponding subsolution of $w$. Then $\wtilde(0,1) = 0$ by the construction, while $\wtilde \geq 0$ in $S$, which follows from a similar property for $w^{(s_+)}$. Note that $w^{(s_+)} \geq 0$ by the maximum principle, since $\check{\eta} \geq d_-(r)$ and the boundary condition \eqref{ws:bot} show that $w^{(s_+)} \geq 0$ along the boundary of $S$. Therefore, $\wtilde \geq 0$ in $S$ as a limit of nonnegative functions. We claim that $\wtilde = 0$ identically in $S$. Indeed, if it is not the case, then the Hopf lemma would give $\wtilde_p(0,1) < 0$; note that $\wtilde$ solve the same elliptic problem and the maximum principle is applicable. But this is in a contradiction with the relation \eqref{wspm}, which holds for $\wtilde$ instead of $\wspm$. The latter identity computed at $q=0$ gives $\wtilde_p(0,1) \geq 0$, since $\wtilde(0,1) = \wtilde_q(0,1) = 0$. Thus, we proved that $\wtilde = 0$ in $S$ and then $\FF = \FF_-(r)$.

A similar argument with subsolutions works for the cases $\hat{\eta} = d_+(r)$ and $\check{\eta} = d_-(r)$.
\end{proof}

\begin{proposition} \label{p:wp}
	Let $h \in C^{2,\gamma}(\overline{S})$ be a solution to \eqref{height} with $r \in (R_c,R_0)$. Then there exists a stream solution $H(p;s)$ with $s \in (s_0,s_-(r))$ such that $\sup_{q \in \R} h_p(q,0) < H_p(0;s)$.
\end{proposition}
\begin{proof}
	If $s_0 = 0$, then the statement is trivial, since then $H_p(0;s) = 1/s \to +\infty$ as $s \to s_0$. If $d_0 = +\infty$, then we choose $s \in (s_0,s_-(r))$ sufficiently small so that $\ws < 0$ on $p=0$. Then $\ws < 0$ everywhere in $S$ by the maximum principle and so $\ws_p < 0$ on $p=0$ by the Hopf lemma. The remaining case $s_0 > 0$ and $d_0 < +\infty$ requires that $H_p(1;s) \to 0$ as $s \to +\infty$ (follows from the classification of vorticity functions in \cite{Kozlov2015}). Then we choose  $s \in (s_0,s_-(r))$ for which $\ws_p < 0$ on $p=1$. This requires $\ws < 0$ on $p=0$ (as follows from the Hopf lemma) and then $\ws_p < 0$ on $p=0$ as before.
\end{proof}

\subsection{Asymptotics for solitary waves}

A solitary wave solution to \eqref{eqn:stream} is defined by an asymptotic relation
\begin{equation} \label{sol}
\lim_{x \to \pm\infty} \eta(x) = d.
\end{equation}
For unidirectional waves it guarantees that the corresponding height function $h$ has a subsolution, which is a laminar flow with the depth $d$. In particular, this requires
\[
d = d_-(r) \ \ \text{or} \ \ d = d_+(r),
\]
where $r$ is the Bernoulli constant. It was recently proved in \cite{KozLokhWheeler2020} that even one-sided, assumption \eqref{sol} requires $d = d_-(r)$; that is all solitary waves are subcritical (supported by subcritical laminar flows $H(p;s)$ with $s>s_c$). Furthermore, one verifies that
\begin{equation} \label{solder}
h(q,p) \to H(p;s_-(r)), \ \ h_q(q,p) \to 0, \ \ q\to \pm \infty
\end{equation}
in $C^{2,\gamma'}([0,1])$ and $C^{1,\gamma'}([0,1])$ respectively for all $\gamma' \in (0,\gamma)$, provided $h \in C^{2,\gamma}(\overline{S})$. Asymptotics \eqref{solder} show that $\FF = \FF_-(r)$, which follows from \eqref{height:ff} by passing to the limit $q \to +\infty$. All these considerations are valid even if we assume that \eqref{sol} holds true only at the positive infinity. 

In order to obtain higher order asymptotics for $h$, we need to introduce the following eigenvalue problem:
\[
- \left( \frac{\varphi_p}{H_p^3} \right)_p = \mu \frac{\varphi}{H_p}, \ \ p \in [0,1],
\]
where $H = H(p;s_+(r))$ and the eigenfunction $\varphi(p)$ is subject to the boundary conditions
\[
\varphi(0) = 0, \ \ \varphi_p(1) = H_p^3(1)\varphi(1).
\]
This Sturm-Liouville problem arises as a form of the dispersion relation; see \cite{ConstantinStrauss04}. The first eigenvalue $\mu_1 = \lambda_1^2 > 0$ is always positive, provided $H = H(p;s_+(r))$ is a supercritical stream solution. This suggests that the difference $h-H$ must decay as $e^{-\lambda_1 q}$ as $q \to +\infty$. A precise statement is given below.

\begin{proposition} \label{solasymp} Let $h \in C^{2,\gamma}(\overline{S})$ be a solution to \eqref{height} and satisfy $\lim_{q \to +\infty} h(q,1) = d_-(r)$. Then
\[
h(q,p) = H(p;s_-(r)) + a \varphi_1(p) e^{-\lambda_1 q} + f(q,p) e^{-\lambda_1'q}, \ \ (q,p) \in S,
\] 
where $a \neq 0$, $\lambda_1' > \lambda_1$ and $f \in C^{2,\gamma}(\overline{S})$.
\end{proposition}

These asymptotics were proved in \cite{Hur07}, under an additional assumption $\lim_{q \to -\infty} h(q,1) = d_-(r)$. However this is not essential and  proofs from \cite{Hur07} are applicable with minor modifications, so we omit it here.

\subsection{Auxiliary functions $\sigma$ and $\kappa$}

For a given $r > R_c$ and $s > s_0$ we define
\begin{equation} \label{sigma}
	\sigma(s;r) = \int_0^1 \left( \frac{1}{2H_p^2(p;s)} - H(p;s) - \Omega(p) + \Omega(1) + r \right) H_p(p;s) \,dp.
\end{equation}
This expression coincides with the flow force constant for $H(p; s)$, but with the Bernoulli constant $R(s)$ replaced by $r$. We also note that
\[
\sigma(s_\mp(r);r) = \FF_\pm(r).
\]
The key property of $\sigma(s; r)$ is stated below.

\begin{lemma} \label{lemma:sigma} For a given $r \in (R_c,R_0)$ the function $s \mapsto \sigma(s;r)$ increases for $s \in (s_0,s_-(r))$, decreases for $s \in (s_-(r),s_+(r))$ and increases to infinity for $s \in (s_+(r),+\infty)$.
\end{lemma}
\begin{proof}
Because
\[
H_p(p;s) = \frac{1}{\sqrt{s^2 - 2 \Omega(p)}}, \ \ \partial_s H_p(p;s) = -s H_p^3(p;s),
\]
we can compute the derivative
\[
\begin{split}
\sigma_s(s;r) & = \int_0^1 \left( \frac{1}{2H_p^2(p;s)} - H(p;s) - \Omega(p) + \Omega(1) + r \right) \partial_s H_p(p;s) \,dp  \\
& +  \int_0^1 \left( -\frac{\partial_s H_p(p;s)}{H_p^3(p;s)} - \partial_s H(p;s) \right) H_p(p;s) \,dp \\
& = \int_0^1 \left( -\frac{1}{2H_p^2(p;s)} - \Omega(p) + \Omega(1) + r \right) \partial_s H_p(p;s) \,dp - d(s)d'(s) \\
& = \int_0^1 \left( -\tfrac12 s^2 + \Omega(1) + r \right) \partial_s H_p(p;s) \,dp - d(s) \int_0^1 \partial_s H_p(p;s) \\
& = -s (r-R(s)) \int_0^1 H_p^3(p;s) \, dp.
\end{split}
\]
Finally, because $R(s) < r$ for $s_-(r) < s < s_+(r)$ and $R(s) > r$ for $s> s_+(r)$ or $s < s_-(r)$ we obtain the statement of the lemma.
\end{proof}

Our function $\sigma(s;r)$ and it's role is similar to the function $\sigma(h)$ introduced by Keady and Norbury in \cite{Keady1975}. The main purpose of the latter is to be used for a comparison with the flow force constant $\FF$.  

The following function will be also involved in our analysis.
\begin{equation} \label{kappa}
	\kappa(s;r) = 2 (\FF - \sigma(s;r)) - (r- R(s))^2.
\end{equation}
A direct computation gives
\[
\begin{split}
\partial_s \kappa(s;r) & = - 2 \partial_s \sigma(s;r) + 2 (r - R(s)) R'(s) \\
& = 2 s (r-R(s)) \int_0^1 H_p^3(p;s) \, dp + 2 (r - R(s)) (s + d'(s)) \\
& = 2 s (r - R(s)).
\end{split}
\]
Thus, we obtain
\begin{lemma} \label{lemma:kappa} For a given $r \in (R_c,R_0)$ the function $s \mapsto \kappa(s;r)$ decreases for $s \in (s_0,s_-(r))$, increases for $s \in (s_-(r),s_+(r))$ and decreases to minus infinity for $s \in (s_+(r),+\infty)$.
\end{lemma}

These monotonicity properties of functions $\sigma$ and $\kappa$ will used in what follows.

\begin{figure}[!t]
	\centering%
	\includegraphics[scale=0.9]{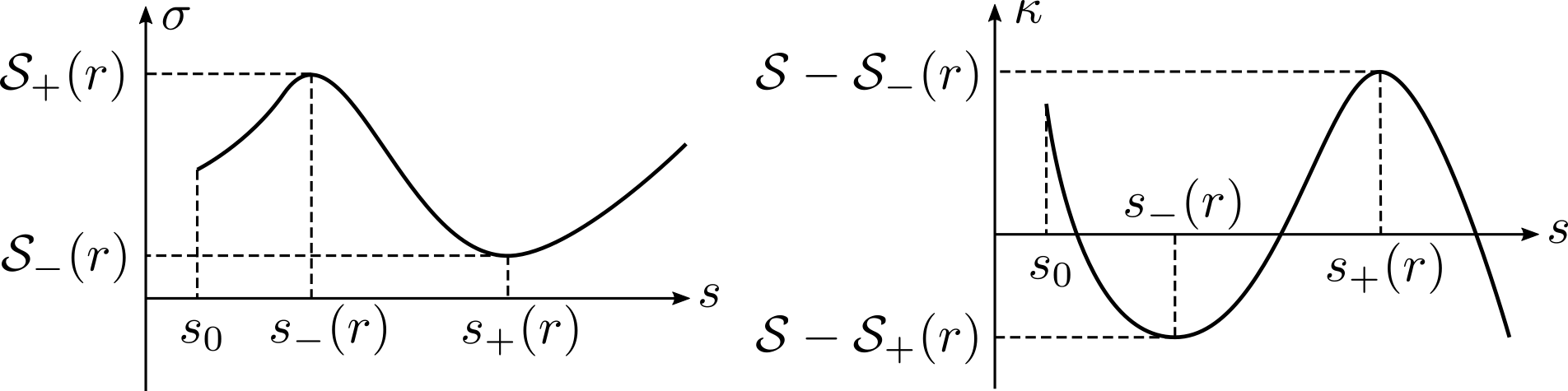}
	\caption{Graphs of functions $\sigma(s;r)$ and $\kappa(s;r)$.}
	\label{fig:stokessolitary}
\end{figure}

\subsection{Flow force flux functions}

Our aim is to extract some information by comparing the flow force constant $\FF$ to $\sigma(s;r)$ for different values of $s > s_0$. For this purpose we introduce the (relative) flow force flux function $\Phi^{(s)}$ by setting
\begin{equation} \label{fff}
\phis(q,p) = \int_0^p \left( \frac{(\ws_p(q,p'))^2}{h_p(q,p') (H_p(p';s))^2} - \frac{(\ws_q(q,p'))^2}{h_p(q,p')} \right) \, dp'.
\end{equation}
The latter functions were recently introduced in \cite{Lokharu2020} and \cite{KozLokhWheeler2020}. The same computation as in Section 3 of \cite{KozLokhWheeler2020} gives
\begin{equation} \label{fff:der}
\phis_q = - \ws_q \left( \frac{1 + (\ws_q)^2}{h_p^2} - \frac{1}{H_p^2}\right), \ \ \phis_p = \frac{(\ws_p)^2}{h_p H_p^2} - \frac{(\ws_q)^2}{h_p}.
\end{equation}
This shows that $\phis \in C^{2,\gamma}(\overline{S})$, provided $h \in C^{2,\gamma}(\overline{S})$ and $\omega \in C^{\gamma}([0,1])$. A key property of $\phis$ is that it solves a homogeneous elliptic equation as stated in the next proposition, while satisfies certain boundary conditions, involving the flow force constant.
\begin{proposition}
	There exist functions $b_1, b_2 \in L^{\infty}(S)$ such that
	\begin{equation} \label{fff:main}
	\frac{1+h_q^2}{h_p^2} \phis_{pp} - 2\frac{h_q}{h_p} \phis_{qp} + \phis_{qq} + b_1 \phis_q + b_2 \phis_p = 0 \ \ \text{in} \ \ S.
	\end{equation}
	Furthermore, $\phis$ satisfies the boundary conditions
\begin{subequations}\label{fff:boundary}
	\begin{alignat}{2}
	\phis &= 2(\FF - \sigma(s;r)) - 2 (r - R(s)) \ws(q,1) + (\ws(q,1))^2 &\qquad& \text{for } p=1,\label{fff:top} \\ 
	\phis &= 0&\qquad& \text{for } p=0. \label{fff:bot}
	\end{alignat}
\end{subequations}
	In the irrotational case $b_1,b_2 = 0$ and \eqref{fff:main} is equivalent to the Laplace equation.
\end{proposition}

\begin{proof}
	For the proof we refer to \cite{KozLokhWheeler2020}. Let us outline the derivation of \eqref{fff:top}. For this purpose we compute the difference
	\[
	\begin{split}
	\FF - \sigma(s;r) & = \int_0^1 \left( \frac{1-(\ws_q)^2}{2h_p^2} - \ws - \frac{1}{2 H_p^2} \right) H_p \, dp \\
	& + \int_0^1 \left( \frac{1-(\ws_q)^2}{2h_p^2} - h - \Omega + \Omega(1) + r \right) \ws_p \, dp \\
	& = \int_0^1 \left( \frac{(\ws_p)^2}{2h_p H_p^2} - \frac{(\ws_q)^2}{2h_p} + \ws H_p \right) \, dp \\
	& + \int_0^1 \left( -\frac{1}{2H_p^2} - \ws - H - \Omega + \Omega(1) + r  \right) \ws_p \, dp.
	\end{split}
	\]
	Now using the identity
	\[
	- \Omega(p) + \Omega(1) + R(s) = \frac{1}{2H_p^2} + H(1)
	\]
	and integrating first-order terms, we conclude that
	\[
	2(\FF - \sigma(s;r)) = 2 (r - R(s)) \ws(q,1) - (\ws(q,1))^2 + \int_0^1 \left( \frac{(\ws_p)^2}{h_p H_p^2} - \frac{(\ws_q)^2}{h_p} \right) \, dp.
	\]
\end{proof}

The next proposition borrowed from \cite{Lokharu2020} explains the meaning of the auxiliary function $\kappa(s; r)$.

\begin{proposition} \label{p:fff}
	Let $h \in  C^{2,\gamma}(\overline{S})$ be a solution to \eqref{height} with $r>R_c$. Assume that the flow force flux function $\phis$ for some $s > s_0$ satisfies $\inf_{q \in \R} \phis(q; 1) \leq 0$. Then
	\[
	\inf_{q \in \R} \phis(q; 1) = \kappa(s; r),
	\]
	where $\kappa(s; r)$ is defined by \eqref{kappa}. Furthermore, the infimum is attained over a sequence $\{q_j\}_{j=1}^\infty$ for which $\lim_{j \to +\infty} \ws(q_j,1) = r- R(s)$.
\end{proposition}

For a proof of this statement we refer to Proposition 2.4 in \cite{Lokharu2020}.

\section{Proof of Theorem \ref{thm:BLC}}

First we prove the left inequality in the first statement of Theorem \ref{thm:BLC}. For the rest of the paper we assume that $h$ is the height function corresponding to $\psi$ and $\eta$; see Section 3.1 for details and notations.

\subsection{Proof of the lower bound $\FF \geq \FF_-(r)$ (part (i) of Theorem \ref{thm:BLC})}

Note that the first claim of Proposition \ref{p:bounds} allows us to assume $\etamin> \dmin$, because otherwise $\FF = \FFmin$ and we
have nothing to prove. In fact we are going to prove a stronger statement, which will be used in the proof of part (ii).

\begin{proposition} \label{p:lowerbound} 
Under assumptions of Theorem 2.1 and assuming that $\etamin> \dmin$, we have $\FF > \FFmin$.
\end{proposition}
\begin{proof}
	Assume the contrary, that $\FF \leq \FFmin$. First we show that the function $\Phi^{(s_+(r))}$ for $s = s_+(r)$ is nonnegative in $S$ and that
	\begin{equation} \label{FFinf}
	\inf_{\R} \Phi^{(s_+(r))}(q,1) > 0.
	\end{equation}
	If the latter is not true, then the infimum above is nonpositive and Proposition \ref{p:fff} gives a sequence of points $\{q_j\}_{j=1}^\infty$ for which 
	\[
	\lim_{j \to +\infty} w^{(s_+(r))}(q_j,1) = r - R(s_+(r)) = 0,
	\]
	which contradicts to the assumption $\check{\eta} > d_-(r)$. Thus, we proved \eqref{FFinf}. Now we show that for some $s \in (s_-(r),s_+(r))$ the function $\phis$ attains negative values somewhere along the upper boundary $p=1$. Indeed, by Proposition \ref{p:bounds} we can find $s_\star \in (s_-(r),s_+(r))$ such that
	\[
	H(1;s_\star) = \inf_\R h(q,1),
	\]
	until $\check{\eta} = d_+(r)$ so that $\FF = \FFplus$ by Proposition \ref{p:bounds} (iii) and we have nothing to prove. The choice of $s_\star$ and the boundary relation \eqref{fff:top} show that $\inf_{\R} \phistar (q,1) \leq \FF - \sigma(s_\star;r)$. But according Lemma \ref{lemma:sigma} (see Figure \ref{fig:stokessolitary}) we have $\sigma(s_\star;r) > \FFmin$ and $\FFmin \geq \FF$ by the assumption. Thus, $\inf_{\R} \phistar(q,1) < 0$. Finally, by the continuity we find $s_\dagger \in (s_-(r),s_+(r))$ such that 
	\[
	\inf_{\R}\Phi^{(s_\dagger)}(q,1) = 0.
	\]
	Then $\kappa(s_\dagger;r) = 0$ by Proposition \ref{p:fff} and the definition of $\kappa$ gives $\FF - \sigma(s_\dagger;r) \geq 0$. But $\sigma(s_\star;r) > \FFmin$ by Lemma \ref{lemma:sigma} and we arrive to a contradiction with the assumption $\FFmin \geq \FF$. Therefore, we proved that $\FF > \FFmin$.
\end{proof}

\subsection{Proof of the upper bound $\FF \leq \FFplus$ and part (iii) of Theorem \ref{thm:BLC}.}

We will complete the proof of the first claim (i) and will prove part (iii) of Theorem \ref{thm:BLC} at the same time. More precisely, we will show below that any nontrivial solution (other than a parallel flow) satisfies the strict inequality $\FF < \FFplus$. Assume the contrary, that $\FF \geq \FFplus$. Then we can prove

\begin{lemma} \label{lemma:up} Let $\FF \geq \FFplus$. Then for any $s \in (s_0,s_-(r))$ the function $\phis$ is nonnegative in $S$. 
\end{lemma}
\begin{proof}
	We prove by a contradiction. For a given $s \in (s_0, s_-(r))$ we assume that
	\[
		\inf_\R \phis(q,1) \leq 0.
	\]
	Then $\kappa(s;r) \leq 0$ by Proposition \ref{p:fff}. Thus, Lemma \ref{lemma:kappa} (see Figure \ref{fig:stokessolitary}) gives $\kappa(s_-(r);r) = \FF - \FF_+(r) < 0$, leading to a contradiction. Therefore, function $\phis$ is positive along the upper boundary $p=1$, zero on the bottom and we finish the proof by using a maximum principle.
\end{proof}

Now we can easily obtain a contradiction. It follows from Lemma \ref{lemma:up} that function $\Phi^{(s_-(r))}$ is nonnegative in $S$, so that 
\[
\Phi^{(s_-(r))}_p(q,0) = \frac{\left(w_p^{(s_-)}(q,0) \right)^2}{h_p(q,0)} > 0
\] 
for all $q\in\R$, provided $\Phi^{(s_-(r))}$ is not zero identically. It is not zero, because otherwise $h = H(p;s_-(r))$ and $\FF = \FF_+(r)$. Thus,  by Proposition \ref{p:wp} we can find a stream solution $H(p; s)$ with $s_0 < s_\star < s_-(r)$ such that $w^{(s_\star)}_p(q_\star,0) = \phistar_p(q_\star,0) = 0$ for some $q_\star \in \R$. But that is in a contradiction with the conclusion of Lemma \ref{lemma:up}, which requires that $\phis_p(q,0) > 0$ for all $q \in \R$ and all $s \in (s_0,s_-(r))$ by the Hopf lemma. This finishes the proof of (iii) and the upper bound in (i).

\subsection{Proof of part (ii) of Theorem \ref{thm:BLC}.}

Let $h$ be an arbitrary solution satisfying assumptions of the theorem and such that $\FF = \FF_-(r)$. On a basis of Proposition \ref{p:lowerbound} we can additionally assume that  $\check{\eta}= d_-(r)$. Our aim is to prove that the function
\[
w(q,p) = h(q,p) - H(p;s_+(r))
\]
uniformly converges to zero as $q \to \pm \infty$. Having that it is left to apply Theorem 3.1 in \cite{Hur07}, which proves that $h$ determines a supercritical and symmetric solitary wave, monotonically decaying
to its asymptotic level $d_-(r)$ on each side of the crest. Without loss of generality, we only need to show that
\begin{equation}\label{wdecay}
w(q,1) \to 0 \ \ \text{as} \ \ q \to +\infty.
\end{equation}
Let us prove
\begin{lemma}\label{lemma:decay}
	Assuming \eqref{wdecay} is not true there exists a non-zero subsolution $\hat{w}$ of $w$ corresponding to a sequence $\{\hat{q}_j\}_{j=1}^\infty$ accumulating at the positive infinity and such that $\hat{w}(q, 1) \to 0$ as $q \to +\infty$. 
\end{lemma}
\begin{proof}
	First we note that even so \eqref{wdecay} is violated there exists a sequence $q_j \to +\infty$ such that $w(q_j,1) \to 0$ as $j \to +\infty$. Indeed, otherwise we can find a subsolution $\tilde{w}$ for which $\inf_\R w(q,1) > 0$ and then $\FF > \FF_-(r)$ by Proposition \ref{p:lowerbound}, leading to a contradiction. On the other hand, for any sufficiently small $\epsilon > 0$ we have
	\[
	\limsup_{q \to +\infty} w(q,1) > \epsilon.
	\]
	Let $I_j = (a_j, b_j)$ be the largest interval containing $q_j$ and such that $w(q; 1) < \epsilon$ for all $q \in I_j$. Assuming $0 < \epsilon < w(1, 0)$, we see that all $I_j$ are bounded, while $|I_j| \to +\infty$ as $j \to +\infty$. Indeed, if lengths are bounded for some subsequence, then we can find a subsolution $\tilde{w}$ corresponding to a subsequence of $\{q_j\}_{j=1}^\infty$, which is not zero identically and such that $\tilde{w}(0,1) = 0$. The latter is forbidden by Proposition \ref{p:bounds} (i). 
	
	Now, let us consider a subsolution $\tilde{w}$ corresponding to the sequence $\{a_j\}_{j=1}^\infty$, where $a_j$ is the left endpoint of $I_j$. Then $\tilde{w}$ must be nonnegative in $S$ and $\tilde{w}(0, 1) = \epsilon$, so it is not zero identically. Moreover, because $|I_j| \to + \infty$ as $j \to +\infty$ and $w(q, 1) \leq \epsilon$ on each $I_j$, we conclude that
	\[
	\tilde{w}(q,1) \leq \epsilon \ \ \text{for all} \ \ q > 0.
	\]
	If $\tilde{w}(q,1)$ converges to zero as $q \to +\infty$ then we are done. If it is not true, then there exists $\hat{\epsilon} > 0$  and a sequence $\hat{q}_j \to +\infty$ such that $\tilde{w}(\hat{q}_j,1) \geq \hat{\epsilon}$, $j \geq 1$. Let $\hat{w}$ be a subsolution of $\tilde{w}$ corresponding to $\{\hat{q}_j\}_{j=1}^\infty$. Then $\hat{w}$ is not zero identically, since $\hat{w}(0,1) \geq \hat{\epsilon}$ by the construction. Furthermore, 
	\[
	\hat{w}(q,1) \leq \epsilon \ \ \text{for all} \ \ q \in \R.
	\]
	Thus, if $\epsilon$ is sufficiently small from the beginning (depending on $r$), then $\hat{w}$ is a small-amplitude solution, supported by a supercritical flow. Therefore, $\hat{w}$ describes a supercritical solitary wave, as follows from \cite{GrovesWahlen08} (the only small-amplitude solutions are solitary waves). As a solitary wave, $\hat{w}$ decays to zero at infinity and provides the desired subsolution; see Proposition \ref{subsub}.
\end{proof}

Let us prove \eqref{wdecay}, If it is not the case, then Lemma \ref{lemma:decay} provides with a non-zero subsolution $\hat{w}$ corresponding to a sequence $\hat{q}_j \to +\infty$ that decays to zero. 
Let $\hat{\Phi}$ be the flow force flux function corresponding to $\hat{w}$. It is defined as a limit over compact subsets of flow force flux functions $\Phi(q+\hat{q}_j,p)$ as $j \to +\infty$, where $\Phi$ is the flow force flux function for $w$. At the same time $\hat{\Phi}$ can be defined explicitly by formula \eqref{fff}, where $\ws$ is replaced by $\hat{w}$ and $h$ is replaced by $\hat{w} + H$. Let us obtain asymptotics for $\hat{\Phi}$. By Proposition \ref{solasymp} function $\hat{w}$ is subject to
\[
\hat{w}(q,p) =  a \varphi_1(p) e^{-\lambda_1 q} + f(q,p) e^{-\lambda_1'q}, \ \ (q,p) \in S,
\]
where $a \neq 0$ and $f \in C^{2,\gamma}(\overline{S})$. Using this formula in \eqref{fff} one obtains
\[
\hat{\Phi}(q,p) = a^2 \frac{\varphi_1'(p)\varphi(p)}{H_p^3(p;s_-(r))} e^{-2\lambda_1 q} + g(q,p) e^{-(\lambda_1 + \lambda_1')q}, \ \ (q,p) \in S,
\]
where $g \in C^{2,\gamma}(\overline{S})$. In particular, there exist $\hat{q}_\star > 0$ and $A>0$ such that
\[
	\hat{\Phi}(q,p) \geq A p e^{-2\lambda_1 q}
\]
for all $q \geq \hat{q}_\star$ and $p \in [0,1]$. Now we consider an interval $\hat{I} = [0,2\hat{q}_\star]$ containing $\hat{q}_\star$. We know that functions $\Phi(q+\hat{q}_j,p)$ converge to $\hat{\Phi}$ in $C^{2,\gamma'}(\hat{I} \times [0,1])$ as $j \to +\infty$. Therefore, we can find an integer $j_\star > 0$ such that
\[
\Phi(\hat{q}_\star + \hat{q}_j,p) \geq \tfrac12 A p e^{-2\lambda_1 \hat{q}_\star}
\]
for all $p \in [0,1]$ and all $j \geq j_\star$. Note that the right-hand side is independent of $j$. Let us put $Q^{(j)} = [\hat{q}_\star + \hat{q}_j, \hat{q}_\star + \hat{q}_{j+1}] \times [0,1]$ and denote by $Q^{(j)}_l$, $Q^{(j)}_r$ and $Q^{(j)}_t$ the left, right and top boundaries of the rectangle, excluding corner points. Furthermore, we put
\[
\check{\eta}_j = d_-(r)+\min_{ [\hat{q}_\star + \hat{q}_j, \hat{q}_\star + \hat{q}_{j+1}]} w(q,1).
\]
By assumption $\check{\eta} = d_-(r)$ we have that $\check{\eta}_j \to 0$ as $j \to +\infty$. Moreover, since $\Phi = w^2$ along the top boundary $p=1$, we also have
\[
\lim_{j \to \infty} \inf_{Q^{(j)}} \Phi = 0.
\]
\begin{figure}[!t]
	\centering%
	\includegraphics[scale=0.7]{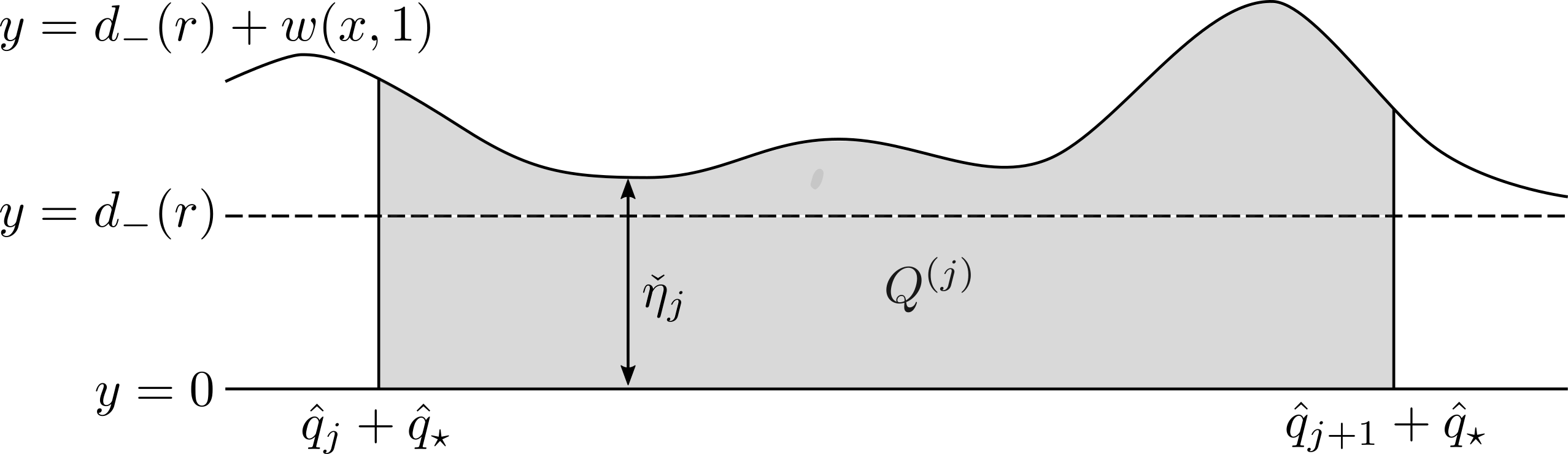}
	\caption{A sketch of the fluid domain in $x,y$-coordinates, corresponding to regions $Q^{(j)}$.}
	\label{fig:stokessolitary}
\end{figure}
Let us consider a stream solution $H(p;s_j)$ for which $d(s_j) = \check{\eta}_j$. Because $\check{\eta}_j > d_-(r)$, we have $s_-(r)<s_j < s_+(r)$, while $s_j \to s_+(r)$ as $j \to +\infty$. In particular, functions $\Phi^{(s_j)}$ converge to $\Phi$ in $C^{2,\gamma}(\overline{S})$ as $j \to +\infty$. This allows us to find an integer $j_\dagger \geq j_\star$ such that
\begin{equation} \label{phipasymp}
\Phi^{(s)}(\hat{q}_\star + \hat{q}_j,p) \geq \tfrac14 A p e^{-2\lambda_1 \hat{q}_\star}
\end{equation}
for all $p \in [0,1]$, all $s \in [s_j,s_+(r)]$ and all $j \geq j_\dagger$. This shows that function $\Phi^{(s)}$ is positive on $Q^{(j)}_l$ and $Q^{(j)}_l$ for all $j \geq j_\dagger$. On the other hand, by the choice of $s_j$, we have 
\[
\min_{Q^{(j)}_t} w^{(s_j)} = 0, \ \ j \geq j_\dagger.
\]
We recall that \eqref{fff:top} gives
\[
\Phi^{(s_j)}(q,1) = 2 (\FF - \sigma(s_j;r)) - 2 (r - R(s))w^{(s_j)}(q,1) + [w^{(s_j)}(q,1)]^2, \ \ q\in\R.
\]
This shows that
\[
\min_{Q^{(j)}_t} \Phi^{(s_j)} \leq 2 (\FF - \sigma(s_j;r)) = 2 (\FF_-(r) - \sigma(s_j;r)) < 0, \ \ j \geq j_\dagger
\]
by Lemma \ref{lemma:sigma}. On the other hand $\min_{Q^{(j)}_t} \Phi > 0$ and by the continuity there exists $s_j^\star \in (s_j,s_+(r))$ such that
\[
\min_{Q^{(j)}_t} \Phi^{(s_j^\star)} = 0.
\]
Because of \eqref{phipasymp} function $\Phi^{(s_j^\star)}$ is positive on the vertical sides $Q^{(j)}_l$ and $Q^{(j)}_r$, while equals to zero on the bottom of the rectangle $Q^{(j)}$, provided $j \geq j_\dagger$. Thus the minimum of $\Phi^{(s_j^\star)}$ on $Q^{(j)}$ is attained for some $(q_j^\star,1) \in Q^{(j)}_t$ on the upper boundary, where $\Phi^{(s_j^\star)}(q_j^\star,1) = 0$. In particular, we have $\Phi^{(s_j^\star)}_p(q_j^\star,1) > 0$ by the Hopf lemma so that $w^{(s_j^\star)}_q(q_j^\star,1) \neq 0$. Now since $\Phi^{(s_j^\star)}_q(q_j^\star,1) = 0$ the boundary relation \eqref{fff:top} after the differentiation with respect to $q$ gives  $w^{(s_j^\star)}(q_j^\star,1) = r - R(s_j^\star)$. Using that we find
\[
\Phi^{(s_j^\star)}(q_j^\star,1) = \kappa(s_j^\star;r) = 0.
\]
But $\kappa(s_j^\star;r) = 0$ requires $\kappa(s_+(r);r) > 0$ by Lemma \ref{lemma:kappa}, which leads to a contradiction, since $\kappa(s_+(r);r) = \FF - \FF_-(r) = 0$. Thus, we proved \eqref{wdecay} and the statement (iii) of the theorem now follows from Theorem 3.1 in \cite{Hur07}.

\bibliographystyle{siam}
\bibliography{bibliography}
\end{document}